\documentclass[12pt]{amsart}
\usepackage{amsfonts}
\usepackage{amssymb}
\usepackage{enumerate}
\usepackage{graphicx}
\usepackage[dvipsnames]{xcolor}

\makeatletter
\@namedef{subjclassname@2010}{  \textup{2010} Mathematics Subject Classification}
\makeatother
\newtheorem{theorem}{Theorem}[section]
\newtheorem{lemma}[theorem]{Lemma}
\newtheorem{proposition}[theorem]{Proposition}
\newtheorem{corollary}[theorem]{Corollary}

\theoremstyle{definition}

\numberwithin{equation}{section}
\DeclareMathOperator{\diam}{diam}

\DeclareMathOperator{\co}{co}

\makeatletter
\@namedef{subjclassname@2020}{  \textup{2020} Mathematics Subject Classification}
\makeatother

\begin{document}
\title[On the Karlsson-Nussbaum conjecture]{On the Karlsson-Nussbaum
conjecture for resolvents of nonexpansive mappings}
\author[A. Huczek]{Aleksandra Huczek}
\address{Department of Mathematics, Pedagogical University of Krakow,
PL-30-084 Cracow, Poland}
\email{aleksandra.huczek@up.krakow.pl}
\author[A. Wi\'{s}nicki]{Andrzej Wi\'{s}nicki}
\address{Department of Mathematics, Pedagogical University of Krakow,
PL-30-084 Cracow, Poland}
\email{andrzej.wisnicki@up.krakow.pl}
\date{}

\begin{abstract}
Let $D\subset \mathbb{R}^{n}$ be a bounded convex domain and $F:D\rightarrow
D$ a $1$-Lipschitz mapping with respect to the Hilbert metric $d$ on $D$
satisfying condition $d(sx+(1-s)y,sz+(1-s)w)\leq \max \{d(x,z),d(y,w) \}$. We
show that if $F$ does not have fixed points, then the convex hull of the
accumulation points (in the norm topology) of the family $\{R_{\lambda
}\}_{\lambda >0}$ of resolvents of $F$ is a subset of $\partial D.$ As a
consequence, we show a Wolff-Denjoy type theorem for resolvents of
nonexpansive mappings acting on an ellipsoid $D$.
\end{abstract}

\maketitle

\section{Introduction}

The study of dynamics of nonlinear mappings started by considering iterates
of holomorphic mappings on one-dimensional bounded domains. In this field,
one of the first theorem is the classical Wolff--Denjoy theorem which
describes dynamics of iteration of holomorphic self-mappings on the unit
disc of the complex plane. It asserts that if $f:\Delta \rightarrow \Delta $
is a holomorphic map of the unit disc $\Delta \subset \mathbb{C}$ without a
fixed point, then there is a point $\xi \in \partial \Delta $ such that the
iterates $f^{n}$ converge locally uniformly to $\xi $ on $\Delta $.
Generalizations of this theorem in different directions have been obtained
by numerous authors (see \cite{Ab,Bu1,Ka1,Nu,LLNW} and references therein).
One such generalization was formulated by Beardon who noticed that the
Wolff--Denjoy theorem can be considered in a purely geometric way depending
only on the hyperbolic properties of a metric and gave its proof using
geometric methods (see \cite{Be1}). In \cite{Be2}, Beardon extended his
approach for strictly convex bounded domains with the Hilbert metric.
Considering the notion of the omega limit set $\omega _{f}(x)$ as the set of
accumulation points of the sequence $f^{n}(x)$ and the notion of the
attractor $\Omega _{f}=\bigcup_{x\in D}\omega _{f}(x)$, we can formulate a
generalization of the Wolff-Denjoy theorem known as the Karlsson-Nussbaum
conjecture, which was formulated independently by Karlsson and Nussbaum (see 
\cite{Ka3, Nu}). This conjecture states that if $D$ is a bounded convex
domain in a finite-dimensional real vector space and $f:D\rightarrow D$ is a
fixed point free nonexpansive mapping acting on the Hilbert metric space $(D,d_{H})$,
then there exists a convex set $\Omega \subseteq \partial D$ such that for
each $x\in D$, all accumulation points $\omega _{f}(x)$ of the orbit $O(x,f)$
lie in $\Omega $.

The aim of this note is to show a variant of the Karlsson-Nussbaum
conjecture for resolvents of nonexpansive ($1$-Lipschitz) mappings. For this
purpose we construct in Section 3 the family of resolvents of a nonexpansive
mapping and prove its main properties: nonexpansivity and the resolvent
identity.
In the literature, the resolvents usually occur in the context of Banach spaces or geodesic spaces that are Busemann convex, see e.g., \cite{ALL, Si}. Then their construction is based on the Banach contraction principle. Since a Hilbert metric space $(D,d_{H})$ is in general not Busemann convex, our construction of resolvents is a little more complicated and exploits the argument related to Edelstein's theorem \cite{Ed}.       

In Section 4 we formulate and prove the main theorem of this work.
We show that if $D\subset \mathbb{R}^{n}$ is a bounded convex domain and $%
F:D\rightarrow D$ is a fixed point free nonexpansive mapping with respect to
the Hilbert metric $d_{H}$ on $D$ satisfying condition 
\begin{equation}
d_{H}(sx+(1-s)y,sz+(1-s)w)\leq \max \{d_{H}(x,z),d_{H}(y,w)\},  \tag{D}
\label{D}
\end{equation}%
then the convex hull of the accumulation points of the family $\{R_{\lambda
}\}_{\lambda >0}$ of resolvents of $F$ is a subset of $\partial D$. Since a
Hilbert metric space $(D,d_{H})$ is Busemann convex if and only if $D$ is
ellipsoid, we obtain as a corollary a Wolff-Denjoy type theorem for
resolvents of nonexpansive mappings acting on an ellipsoid $D$.

\section{Preliminaries}

Let $V$ be a finite dimensional real vector space, $D\subset V$ a convex
bounded domain and $(D,d)$ a metric space. A curve $\sigma :[a,b]\rightarrow
D$ is said to be \textit{geodesic }if $d(\sigma (t_{1}),\sigma
(t_{2}))=\left\vert t_{1}-t_{2}\right\vert $ for all $t_{1},t_{2}\in \lbrack
a,b]$. We will use the same name for the image $\sigma ([a,b])\subset D$ of $%
\sigma $, denoted by $[\sigma (a),\sigma (b)].$ We say that $D$ is a \textit{%
geodesic space} if every two points of $D$ can be joined by a geodesic. A
map $F:D\rightarrow D$ is called \textit{contractive} if $%
d(F(x),F(y))<d(x,y) $ for any distinct points $x,y\in D$. A map $%
F:D\rightarrow D$ is called \textit{nonexpansive} if for any $x,y\in D,$ $%
d(F(x),F(y))\leq d(x,y)$.

We recall the definition of \textit{the Hilbert metric space}. If $x,y\in D$%
, consider the straight line passing through $x$ and $y$ that intersects the
boundary of $D$ in precisely two points $a$ and $b$. Assuming that $x$ is
between $a$ and $y$, and $y$ is between $x$ and $b$, we define the
cross-ratio metric 
\begin{equation*}
d_{H}(x,y)=\log \bigg(\frac{||y-a||\,||x-b||}{||x-a||\,||y-b||}\bigg),\qquad
x\neq y.
\end{equation*}%
Furthermore, we put $d_{H}(x,y)=0$ if $x=y$.

Following Beardon \cite{Be2} we consider the subsequent lemmas.

\begin{lemma}
\label{hilbert1}Let $D_{1},D_{2}\subset V,\,D_{1}\subset D_{2}$ be bounded
convex domains and $(D_{1},d_{1}),(D_{2},d_{2})$ be Hilbert metric spaces,
then $d_{2}\leq d_{1}$. Furthermore, for distinct points $x,y\in
D_{1},d_{1}(x,y)=d_{2}(x,y)$ iff the segment $L_{xy}\cap D_{1}$ coincides
with $L_{xy}\cap D_{2}$.
\end{lemma}

\begin{lemma}
\label{Bcontraction}Suppose that $(D,d_{H})$ is a Hilbert metric space, $%
x_{0}\in D$ and $l\in \lbrack 0,1)$. Then the mapping $g(x)=x_{0}+l(x-x_{0})$
is contractive.
\end{lemma}

\begin{proof}
Fix $x_{0}\in D$ and $l\in \lbrack 0,1)$. Let $x,y\in D$ and consider the
straight line passing through $x$ and $y$ that intersects $\partial D$ in
two points $x^{\prime }$ and $y^{\prime }$ such that $x$ is between $%
x^{\prime }$ and $y$, and $y$ is between $x$ and $y^{\prime }$. Take two
points $z^{\prime }=(1-l)x_{0}+lx^{\prime }\in \partial g(D),w^{\prime
}=(1-l)x_{0}+ly^{\prime }\in \partial g(D)$, and note that the points $%
z^{\prime },g(x),g(y),w^{\prime }$ are collinear such that $g(x)$ is between 
$z^{\prime }$ and $g(y)$, and $g(y)$ is between $g(x)$ and $w^{\prime }$.
Since $g(D)$ lies in a compact subset of $D$, it follows from Lemma \ref%
{hilbert1} that $d_{H}(g(x),g(y))<d_{2}(g(x),g(y))$, where $d_{2}$ denotes
the Hilbert metric in $g(D)$. By definition of the Hilbert metric space we
have 
\begin{equation*}
d_{H}(x,y)=\log \bigg(\frac{||x^{\prime }-y||\,||x-y^{\prime }||}{%
||x^{\prime }-x||\,||y-y^{\prime }||}\bigg)=\log \bigg(\frac{||z^{\prime
}-g(y)||\,||w^{\prime }-g(x)||}{||z^{\prime }-g(x)||\,||w^{\prime }-g(y)||}%
\bigg)=d_{2}(g(x),g(y)).
\end{equation*}%
Therefore we get $d_{H}(g(x),g(y))<d_{H}(x,y)$.
\end{proof}

Note that if $D\subset V$ is a bounded convex domain, then the Hilbert
metric $d_{H}$ is locally equivalent to the euclidean norm in $V$.
Furthermore, for any $w\in D$, if $\{x_{n}\}$ is a sequence in $D$
converging to $\xi \in \partial D=\overline{D}\setminus D$, then 
\begin{equation*}
d_{H}(x_{n},w)\rightarrow \infty
\end{equation*}%
(see \cite{Be2, Ka1}). The above property is equivalent to properness of $D$%
, that is, every closed and bounded subset of $(D,d_{H})$ is compact. It is
not difficult to show that for $x,y,z\in D$ and $s\in \lbrack 0,1],$ 
\begin{equation}
d_{H}(sx+(1-s)y,z)\leq \max \{d_{H}(x,z),d_{H}(y,z)\}.  \tag{C}  \label{C}
\end{equation}%
In what follows, we will assume a more restrictive condition: for all $%
x,y,z,w\in D$ and $s\in \lbrack 0,1],$ 
\begin{equation}
d_{H}(sx+(1-s)y,sz+(1-s)w)\leq \max \{d_{H}(x,z),d_{H}(y,w)\}.  \tag{D}
\end{equation}

\section{Resolvents of nonexpansive mappings}

In this section we describe the construction of a resolvent of a
nonexpansive mapping acting on a Hilbert metric space. Let $D\subset V$ be a
convex bounded domain and $F:D\rightarrow D$ a nonexpansive mapping with
respect to the Hilbert metric $d$ on $D$. Recall that the topology of $(D,d)$ coincides with the Euclidean topology and $(D,d)$ is proper metric space,
that is, every closed ball $\bar{B}(x_0,r), \, x_0 \in D,$ is compact. We fix $x\in D$, $\lambda >0$, and
define a mapping 
\begin{equation*}
G_{x,\lambda }(y)=\frac{1}{1+\lambda }x+\frac{\lambda }{1+\lambda }%
F(y),\qquad y\in D.
\end{equation*}%
It follows from Lemma \ref{Bcontraction} that $G_{x,\lambda }$ is
contractive.

We show that $G_{x,\lambda }(D)$ is bounded. For this purpose, we
select $w\in G_{x,\lambda }(D)$. Note that there exists $y\in D$ such that $%
w=\frac{1}{1+\lambda }x+\frac{\lambda }{1+\lambda }F(y)$. We show that $B(w,%
\frac{1}{1+\lambda }d)\subset D$, where $d=\inf_{v\in \partial D}||v-x||$.
Choose any $w^{\prime }\in B(w,\frac{1}{1+\lambda }d)$. Then there exists $%
z\in D$ such that $w^{\prime }=\frac{1}{1+\lambda }z+\frac{\lambda }{%
1+\lambda }F(y)$. Note that 
\begin{equation*}
||w-w^{\prime }||=\bigg|\bigg|\frac{1}{1+\lambda }x+\frac{\lambda }{%
1+\lambda }F(y)-\frac{1}{1+\lambda }z+\frac{\lambda }{1+\lambda }F(y)\bigg|%
\bigg|=\frac{1}{1+\lambda }||x-z||.
\end{equation*}%
If $||w-w^{\prime }||<\frac{1}{1+\lambda }$, then $||x-z||=(1+\lambda
)||w-w^{\prime }||<d$. It implies that $z\in D$ and hence $w^{\prime }\in D$%
. It follows that for all $w\in G_{x,\lambda }(D),$ 
\begin{equation}\label{ineqinf} 
\inf_{v \in \partial D}||v-w||\geq \frac{1}{1+\lambda }d.
\end{equation}%
Take a sequence $\{w_n \} \subset G_{x, \lambda}(D)$. Since $\overline{D}$ is compact in the Euclidean topology, there exists a subsequence $\{w_{n_k} \}$ and $x_0 \in \overline{D}$ such that $||w_{n_k}-x_0|| \rightarrow 0$, if $k \rightarrow \infty$. It follows from (\ref{ineqinf}) that $x_0 \in D$, and hence $d(w_{n_k},x_0) \rightarrow 0$ since the topology of $(D,d)$ coincides with the Euclidean topology. Therefore, $G_{x, \lambda}(D)$ is bounded in $(D,d)$ and by properness of $D$ we have that $\overline{G_{x,\lambda }(D)}$ is compact in $(D,d)$.

Note that $D\supset G_{x,\lambda }(D)\supset G_{x,\lambda }^{2}(D)\supset
... $, which means that the orbits of $G_{x,\lambda }$ are bounded. Fix $%
y\in D$. Since $\overline{G_{x,\lambda }(D)}$ is compact, there exists a
subsequence $\{G_{x,\lambda }^{n_{k}}(y)\}$ of $\{G_{x,\lambda }^{n}(y)\}$
converging to some $z\in D$. Let 
\begin{equation*}
d_{n}=d(G_{x,\lambda }^{n}(y),G_{x,\lambda }^{n+1}(y)).
\end{equation*}%
Since $G_{x,\lambda }$ is contractive, the sequence $\{d_{n}\}$ is
decreasing and hence it converges to some $\zeta $, as $n\rightarrow \infty $%
. Hence%
\begin{equation*}
\zeta \leftarrow d_{n_{k}}=d(G_{x,\lambda }^{n_{k}}(y),G_{x,\lambda
}^{n_{k}+1}(y))\rightarrow d(G_{x,\lambda }(z),z),
\end{equation*}%
and 
\begin{equation*}
\zeta \leftarrow d_{n_{k}+1}=d(G_{x,\lambda }^{n_{k}+1}(y),G_{x,\lambda
}^{n_{k}+2}(y))\rightarrow d(G_{x,\lambda }^{2}(z),G_{x,\lambda }(z)).
\end{equation*}%
We get%
\begin{equation*}
d(G_{x,\lambda }^{2}(z),G_{x,\lambda }(z))=d(G_{x,\lambda }(z),z)=\zeta .
\end{equation*}%
Since the map $G_{x,\lambda }$ is contractive, $G_{x,\lambda }(z)=z.$
Moreover, $z$ is the unique fixed point of $G_{x,\lambda }$. Indeed, otherwise if $z_1, z_2 \in D$, $z_1 \neq z_2$ are fixed points of $G_{x, \lambda}$, then 

\begin{equation*}
d(z_{1},z_{2})=d(G_{x,\lambda }(z_{1}),G_{x,\lambda }(z_{2}))<d(z_{1},z_{2}),
\end{equation*}%
and we obtain a contradiction. Define $z=R_{\lambda }(x).$ We refer to the
mapping $R_{\lambda }:D\rightarrow D$ as \textit{the resolvent of }$F.$ We
have%
\begin{equation*}
z=G_{x,\lambda }(z)=\frac{1}{1+\lambda }x+\frac{\lambda }{1+\lambda }%
F(z),\qquad x\in D,\;\lambda >0,
\end{equation*}%
and hence%
\begin{equation}
R_{\lambda }(x)=\frac{1}{1+\lambda }x+\frac{\lambda }{1+\lambda }%
F(R_{\lambda }(x)),\qquad x\in D,\,\lambda >0.  \label{res}
\end{equation}%
Furthermore, any converging subsequence $G_{x,\lambda }^{m_{k}}(y)$ has the
limit $z$ (the unique fixed point), as $k\rightarrow \infty .$ This gives
the formula: 
\begin{equation}
\lim_{n\rightarrow \infty }G_{x,\lambda }^{n}(y)=R_{\lambda }(x),\qquad y\in
D.  \label{res2}
\end{equation}

It turns out that if $(D,d)$ is sufficiently regular, then the resolvent of
a nonexpansive mapping is also nonexpansive.

\begin{lemma}
\label{nonex}Let $(D,d)$ be a Hilbert metric space satisfying condition (\ref%
{D}), $F:D\rightarrow D$ a nonexpansive mapping, and $\lambda >0$. Then the
resolvent $R_{\lambda }:D\rightarrow D$ is nonexpansive.
\end{lemma}

\begin{proof}
Fix $z_{0},z_{1},z_{2}\in D$. First we show that $d(G_{z_{1},\lambda
}^{n}(z_{0}),G_{z_{2},\lambda }^{n}(z_{0}))\leq d(z_{1},z_{2})$ for each $n$%
. We proceed by induction. For $n=1$, it follows from condition (\ref{C})
that 
\begin{eqnarray*}
d(G_{z_{1},\lambda }(z_{0}),G_{z_{2},\lambda }(z_{0})) &=&d\bigg(\frac{1}{%
1+\lambda }z_{1}+\frac{\lambda }{1+\lambda }F(z_{0}),\frac{1}{1+\lambda }%
z_{2}+\frac{\lambda }{1+\lambda }F(z_{0})\bigg) \\
&\leq &\max \{d(z_{1},z_{2}),d(F(z_{0}),F(z_{0}))\}=d(z_{1},z_{2}).
\end{eqnarray*}%
Fix $n\in \mathbb{N}$ and suppose that $d(G_{z_{1},\lambda
}^{n}(z_{0}),G_{z_{2},\lambda }^{n}(z_{0}))\leq d(z_{1},z_{2})$. Then it
follows from (\ref{D}) that 
\begin{eqnarray*}
d(G_{z_{1},\lambda }^{n+1}(z_{0}),G_{z_{2},\lambda }^{n+1}(z_{0})) &=&d\bigg(%
\frac{1}{1+\lambda }z_{1}+\frac{\lambda }{1+\lambda }F(G_{z_{1},\lambda
}^{n}(z_{0})),\frac{1}{1+\lambda }z_{2}+\frac{\lambda }{1+\lambda }%
F(G_{z_{2},\lambda }^{n}(z_{0}))\bigg) \\
&\leq &\max \{d(z_{1},z_{2}),d(G_{z_{1},\lambda
}^{n}(z_{0}),G_{z_{2},\lambda }^{n}(z_{0}))\}=d(z_{1},z_{2}).
\end{eqnarray*}%
Now the formula (\ref{res2}) yields 
\begin{eqnarray*}
d(R_{\lambda }(z_{1}),R_{\lambda }(z_{2})) &=&\lim_{n\rightarrow \infty }d%
\bigg(G_{z_{1},\lambda }^{n}(z_{0}),G_{z_{2},\lambda }^{n}(z_{0})\bigg) \\
&\leq &d(z_{1},z_{2}),
\end{eqnarray*}%
which shows that $R_{\lambda }$ is a nonexpansive mapping.
\end{proof}

We will also use the following property called the resolvent identity.

\begin{proposition}
\label{resid} Suppose that $F:D\rightarrow D$ is a nonexpansive mapping.
Then its resolvent $R_{\lambda }$ satisfies 
\begin{equation*}
R_{\lambda }(x)=R_{\mu }\bigg(\frac{\lambda -\mu }{\lambda }R_{\lambda }(x)+%
\frac{\mu }{\lambda }x\bigg),\qquad x\in D,
\end{equation*}%
for all $\lambda >\mu >0$.
\end{proposition}

\begin{proof}
Fix $x\in D$ and $\lambda ,\mu >0$ such that $\lambda >\mu $. Define 
\begin{equation}
y:=\frac{\lambda -\mu }{\lambda }R_{\lambda }(x)+\frac{\mu }{\lambda }x.
\label{y}
\end{equation}
It follows from (\ref{res}) that there exists the unique point 
\begin{equation}\label{last2}
z:=R_{\mu }(y)=\frac{1}{1+\mu }y+\frac{\mu }{1+\mu }F(R_{\mu }(y)).
\end{equation}%
On the other hand, we have%
\begin{equation}\label{last1}
\tilde{z} :=R_{\lambda }(x)=\frac{1}{1+\lambda }x+\frac{\lambda }{1+\lambda 
}F(R_{\lambda }(x)).
\end{equation}%
From the above and (\ref{y}) we get $\lambda y-\mu \tilde{z}(1+\lambda
)=(\lambda -\mu )\tilde{z}-\lambda \mu F(\tilde{z})$, which implies 
\begin{equation*}
\tilde{z}=\frac{1}{1+\mu }y+\frac{\mu }{1+\mu }F(\tilde{z}).
\end{equation*}%
Therefore, from the uniqueness of the construction of the point $z$ and by (\ref{last2}) and (\ref{last1}), we have
\begin{equation*}
R_{\mu}(y)= z=\tilde{z} = R_{\lambda}(x).
\end{equation*}
\end{proof}

For any $x\in D$, $F:D\rightarrow D$, the set of accumulation points (in the
norm topology) of the sequence $\{F^{n}(x)\}$ is called the \textit{omega
limit set of }$x$ and is denoted by $\omega _{F}(x)$. In a similar way, if $%
R_{\lambda }:D\rightarrow D,\lambda >0,$ is a family of resolvents of $F$,
we define 
\begin{equation*}
\omega _{\{R_{\lambda } \}_{\lambda>0}} (x)=\{y\in \overline{D}:\left\Vert R_{\lambda
_{n}}(x)-y\right\Vert \rightarrow 0\text{ for some increasing sequence }%
\{\lambda _{n}\},\lambda _{n}\rightarrow \infty \},
\end{equation*}%
and the \textit{attractor }of $\{R_{\lambda }\}_{\lambda>0},$ 
\begin{equation*}
\Omega _{_{\{R_{\lambda }\}_{\lambda>0}}}=\bigcup_{x\in D}\omega _{_{\{R_{\lambda
}\}_{\lambda>0}}}(x).
\end{equation*}%

\begin{lemma}\label{boundary} 
	Suppose that $F:D\rightarrow D$ is a nonexpansive mapping without fixed points and $R_{\lambda}:D \rightarrow D$, $\lambda >0$ is a family of resolvents of $F$. Then $\Omega _{_{\{R_{\lambda }\}_{\lambda>0}}}\subset \partial D.$
\end{lemma}
\begin{proof}
On the contrary, we suppose that there exists $y\in D$ such that $%
\left\Vert R_{\lambda _{n}}(x)-y\right\Vert \rightarrow 0$ for some $x\in D$
and an increasing sequence $\{\lambda _{n}\},\lambda _{n}\rightarrow \infty
. $ Then%
\begin{equation}
||R_{\lambda _{n}}(x)-F(R_{\lambda _{n}}(x))||=\frac{1}{1+\lambda _{n}}%
||x-F(R_{\lambda _{n}}(x))||\rightarrow 0,  \label{1}
\end{equation}%
as $n\rightarrow \infty .$ Since the topology of $(D,d)$ coincides with the
norm topology, $F:D\rightarrow D$ is norm-continuous, and hence 
\begin{equation*}
||F(y)-y||\leq ||F(y)-F(R_{\lambda _{n}}(x))||+||F(R_{\lambda
_{n}}(x))-R_{\lambda _{n}}(x)||+||R_{\lambda _{n}}(x)-y||\rightarrow 0,
\end{equation*}%
as $n\rightarrow \infty $. Thus $F(y)=y$, and we obtain a contradiction.
\end{proof}

\section{Main theorem}

We begin by recalling one of the fundamental properties of a Hilbert metric
space that allows Karlsson and Noskov to extend Beardon's Wolff--Denjoy
theorem for bounded strictly convex domains (see \cite[Theorem 5.5]{KaNo},%
\cite[Proposition 8.3.3]{LeNu}).

\begin{lemma}
\label{Ax2}Let $D\subseteq V$ be an open bounded convex set and $d$ a
Hilbert metric on $D$. If $\{x_{n}\}$ and $\{y_{n}\}$ are convergent
sequences in $D$ with limits $x$ and $y$ in $\partial D$, respectively, and
the segment $[x,y]\nsubseteq \partial D$, then for each $z\in D$ we have 
\begin{equation*}
\lim_{n\rightarrow \infty } [d(x_{n},y_{n})-\max
\{d(x_{n},z),d(y_{n},z)\}]=\infty .
\end{equation*}
\end{lemma}

We also need the following standard argument that can be found for example in 
\cite[Lemma 5.4]{BKR3}.

\begin{lemma}
\label{limh}Let $(D,d)$ be a separable metric space and let $%
a_{n}:D\rightarrow \mathbb{R}$ be a nonexspansive mapping for each $n\in 
\mathbb{N}$. If for every $x\in D$, the sequence $\{a_{n}(x)\}$ is bounded,
then there exists a subsequence $\{a_{n_{j}}\}$ of $\{a_{n}\}$ such that $%
\lim_{j\rightarrow \infty }a_{n_{j}}(x)$ exists for every $x\in D$.
\end{lemma}

Fix $x_{0}\in D$ and consider a sequence $\{x_n \in D : n \in \mathbb{N} \}$ contained in $D$. Define $a_{n}(x)=d(x,x_{n})-d(x_{n},x_{0})$ for any $%
n\in \mathbb{N}$. Note that 
\begin{equation*}
|a_{n}(y)-a_{n}(x)|\leq d(x,y), 
\end{equation*}
i.e., $a_{n}$
is nonexpansive and the sequence $\{a_{n}(y)\}$ is bounded (by $d(y,x_0)$) for every $y\in D$%
. It follows from Lemma \ref{limh} that there exists a subsequence $%
\{x_{n_{j}}\}$ of $\{x_{n}\}$ such that $\lim_{j\rightarrow \infty
}a_{n_{j}}(x)$ exists for any $x\in D$, i.e., 
\begin{equation}
\lim_{j\rightarrow \infty }d(x,x_{n_{j}})-d(x_{n_{j}},x_{0})  \label{el}
\end{equation}%
exists for every $x\in D$.

Now we are in a position to prove a variant of the Karlsson--Nussbaum
conjecture for resolvents of nonexpansive mappings.

\begin{theorem}
\label{main}Let $D\subset V$ be a bounded convex domain. Suppose that $(D,d)$
is a Hilbert metric space satisfying condition (\ref{D}) and $R_{\lambda
}:D\rightarrow D,\lambda >0,$ is a family of resolvents of a nonexpansive
mapping $F:D\rightarrow D$ without fixed points. Then $\co\Omega
_{\{R_{\lambda }\}_{\lambda>0}}\subseteq \partial D$.
\end{theorem}

\begin{proof}
Suppose on the contrary that there exist $z_{1},\ldots ,z_{m}\in D,\zeta
^{1}\in \omega _{\{R_{\lambda }\}_{\lambda>0}}(z_{1}),\ldots ,\zeta ^{m}\in \omega
_{\{R_{\lambda }\}_{\lambda>0}}(z_{m})$ and $0<\alpha _{1},\ldots ,\alpha _{m}<1$ with $%
\sum_{i=1}^{m}\alpha _{i}=1$ such that $\sum_{i=1}^{m}\alpha _{i}\zeta
^{i}\in D$. Since $F$ does not have fixed points, it follows from Lemma \ref{boundary} that the omega limit sets $%
\omega _{\{R_{\lambda }\}_{\lambda>0}}(z_{i})\subseteq \partial D$, $i=1,\ldots ,m,$
and, following \cite[Theorem 8.3.11]{LeNu}, we can assume that $m\geq 2$ is
minimal with the property that $\sum_{i=1}^{m}\alpha _{i}\zeta ^{i}\in D$.
It follows that $R_{\lambda _{j}^{i}}(z_{i})\rightarrow \zeta ^{i}\in
\partial D$ for some increasing sequences $\{\lambda _{j}^{i}\}_{j},\lambda
_{j}^{i}\rightarrow \infty $, as $j\rightarrow \infty ,i=1,\ldots ,m.$ We
put $\zeta =\zeta ^{1}$ and $\eta =\sum_{i=2}^{m}\mu _{i}\zeta ^{i}$, where $%
\mu _{i}=\frac{\alpha _{i}}{1-\alpha _{1}}$ for $i\in \lbrack 2,m]$. Let $%
\eta ^{j}=\sum_{i=2}^{m}\mu _{i}R_{\lambda _{j}^{i}}(z_{i})$ for all $j\geq
1 $. Since $m$ is minimal, we have $\zeta ,\eta \in \partial D$ and $\alpha
_{1}\zeta +(1-\alpha _{1})\eta \in D$. Since $D$ is convex, we get $\alpha
\zeta +(1-\alpha )\eta \in D$ for all $\alpha \in (0,1)$. By passing to a
subsequence we can assume from (\ref{el}) that for every $x\in D$ there
exists the limit 
\begin{equation}
g(x)=\lim_{j\rightarrow \infty }d(x,R_{\lambda
_{j}^{1}}(z_{1}))-d(R_{\lambda _{j}^{1}}(z_{1}),z_{1}).
\end{equation}%
Since 
\begin{equation*}
\bigg|\bigg|y-\frac{\lambda _{j}^{1}-\mu }{\lambda _{j}^{1}}y-\frac{\mu }{%
\lambda _{j}^{1}}z_{1}\bigg|\bigg|=\frac{\mu }{\lambda _{j}^{1}}%
||y-z_{1}||\rightarrow 0,\ \ \ \mbox{as}\;\lambda _{j}^{1}\rightarrow \infty
,
\end{equation*}%
and topologies of $(D,d)$ and $(\overline{D},||\cdot ||)$ coincide on $D$,
we have 
\begin{equation}
d\bigg(y,\frac{\lambda _{j}^{1}-\mu }{\lambda _{j}^{1}}y+\frac{\mu }{\lambda
_{j}^{1}}z_{1}\bigg)\rightarrow 0,  \label{resnorm}
\end{equation}%
if $\lambda _{j}^{1}\rightarrow \infty .$ According to Lemma \ref{nonex},
Proposition \ref{resid}, (\ref{resnorm}) and condition (\ref{C}), we get 
\begin{eqnarray*}
g(R_{\mu }(y)) &=&\lim_{j\rightarrow \infty }d(R_{\mu }(y),R_{\lambda
_{j}^{1}}(z_{1}))-d(R_{\lambda _{j}^{1}}(z_{1}),z_{1}) \\
&=&\lim_{j\rightarrow \infty }d\bigg(R_{\mu }(y),R_{\mu }\bigg(\frac{\lambda
_{j}^{1}-\mu }{\lambda _{j}^{1}}R_{\lambda _{j}^{1}}(z_{1})+\frac{\mu }{%
\lambda _{j}^{1}}z_{1}\bigg)\bigg)-d(R_{\lambda _{j}^{1}}(z_{1}),z_{1}) \\
&\leq &\limsup_{j\rightarrow \infty }d\bigg(y,\frac{\lambda _{j}^{1}-\mu }{%
\lambda _{j}^{1}}R_{\lambda _{j}^{1}}(z_{1})+\frac{\mu }{\lambda _{j}^{1}}%
z_{1}\bigg)-d(R_{\lambda _{j}^{1}}(z_{1}),z_{1}) \\
&=&\lim_{j\rightarrow \infty }d\bigg(\frac{\lambda _{j}^{1}-\mu }{\lambda
_{j}^{1}}y+\frac{\mu }{\lambda _{j}^{1}}z_{1},\frac{\lambda _{j}^{1}-\mu }{%
\lambda _{j}^{1}}R_{\lambda _{j}^{1}}(z_{1})+\frac{\mu }{\lambda _{j}^{1}}%
z_{1}\bigg)-d(R_{\lambda _{j}^{1}}(z_{1}),z_{1}) \\
&\leq &\lim_{j\rightarrow \infty }d(y,R_{\lambda
_{j}^{1}}(z_{1}))-d(R_{\lambda _{j}^{1}}(z_{1}),z_{1}) \\
&=&g(y).
\end{eqnarray*}%
From the above we have $g(R_{\mu }(y))\leq g(y)\leq d(y,z_{1})$ for every $%
y\in D$ and $\mu >0$. It follows from (\ref{C}) that for any $k\in \mathbb{N}
$, 
\begin{equation*}
g(\eta ^{k})=g\bigg(\sum_{i=2}^{m}\,\mu _{i}R_{\lambda _{k}^{i}}(z_{i})\bigg)%
\leq \max_{i=2,\ldots ,m}g(z_{i})\leq \max_{i=2,\ldots ,m}d(z_{i},z_{1})=M.
\end{equation*}%
Consequently, by diagonal method, there exists a subsequence $\lambda
_{j_{1}}^{1}\leq \lambda _{j_{2}}^{1}\leq \ldots \leq \lambda
_{j_{k}}^{1}\leq \ldots $ of $\{\lambda _{j}^{1}\}$ such that 
\begin{equation}
\limsup_{k\rightarrow \infty }d(\eta ^{k},R_{\lambda
_{j_{k}}^{1}}(z_{1}))-d(R_{\lambda _{j_{k}}^{1}}(z_{1}),z_{1})\leq M+1.
\label{cont}
\end{equation}%
Since $R_{\lambda _{j}^{i}}(z_{i})\rightarrow \zeta ^{i}$, as $j\rightarrow
\infty $ for any $i=1,\ldots ,m$, we have 
\begin{equation*}
||\eta ^{j}-\eta ||=\bigg|\bigg|\sum_{i=2}^{m}\mu _{i}R_{\lambda
_{j}^{i}}(z_{i})-\sum_{i=2}^{m}\mu _{i}\zeta ^{i}\bigg|\bigg|\leq
\sum_{i=2}^{m}\mu _{i}||R_{\lambda _{j}^{i}}(z_{i})-\zeta ^{i}||\rightarrow
0,\quad j\rightarrow \infty ,
\end{equation*}%
which implies that $||\eta ^{j}-\eta ||\rightarrow 0,j\rightarrow \infty $.
Moreover, since $[\zeta ,\eta ]\nsubseteq \partial D$ it follows from Lemma %
\ref{Ax2} that 
\begin{equation*}
\liminf_{k\rightarrow \infty }d(\eta ^{k},R_{\lambda
_{j_{k}}^{1}}(z_{1}))-d(R_{\lambda _{j_{k}}^{1}}(z_{1}),z_{1})=\infty .
\end{equation*}%
However, the above formula contradicts (\ref{cont}).
\end{proof}

We can use Theorem \ref{main} to show a Wolff-Denjoy type theorem for
resolvents of nonexpansive mappings. Let $(D,d)$ be a geodesic metric space
and $[x,y],[x^{\prime },y^{\prime }]$ two arbitrary geodesic segments in $D$.
For every $\alpha \in \lbrack 0,1]$, consider the point $z = \alpha x + (1- \alpha)y$ on segment $[x,y]$ such that $d(\alpha x + (1- \alpha)y,y ) = \alpha d(x,y)$ and in the same way, the point $z^{\prime } = \alpha x^{\prime} + (1 - \alpha) y^{\prime}$ on segment $%
[x^{\prime },y^{\prime }]$ such that $d(\alpha x^{\prime} + (1- \alpha)y^{\prime},y^{\prime} ) = \alpha d(x^{\prime},y^{\prime})$. Recall that a geodesic space $%
(D,d)$ is called \textit{Busemann} \textit{convex} if 
\begin{equation*}
d(z,z^{\prime })\leq (1-\alpha )d(x,x^{\prime })+\alpha d(y,y^{\prime })
\end{equation*}%
for every $x,y,x^{\prime },y^{\prime }\in D$ and $\alpha \in \lbrack 0,1]$.

Combining Corollary 3.3 and Proposition 3.4 in \cite{AlKo}, we obtain the
following proposition (see also \cite{KeSt}, \cite[p. 191]{Pa}).

\begin{proposition}
\label{ellips}Let $D\subset V$ be a bounded convex domain. A Hilbert metric
space $(D,d)$ is Busemann convex if and only if $D$ is an ellipsoid.
\end{proposition}

Since in Hilbert's metric spaces every straight-line segment is a geodesic,
it follows from Proposition \ref{ellips} that $(D,d)$ satisfies condition (%
\ref{D}), whenever $D$ is an ellipsoid. This leads to the following
Wolff-Denjoy type theorem for resolvents of nonexpansive mappings.

\begin{corollary}
Suppose that $D\subset V$ is an ellipsoid and $R_{\lambda }:D\rightarrow
D,\lambda >0,$ is the resolvent of a nonexpansive mapping $F:D\rightarrow D$
(with respect to Hilbert's metric) without fixed points. Then there exists $%
\xi \in \partial D$ such that $\{R_{\lambda } \}_{\lambda>0}$ converge uniformly on bounded
sets of $D$ to $\xi $.
\end{corollary}

\begin{proof}
It follows from Theorem \ref{main} that $\co\Omega _{\{R_{\lambda
}\}_{\lambda>0}}\subseteq \partial D.$ Since $D$ is strictly convex, $\Omega
_{\{R_{\lambda }\}_{\lambda>0}}$ consists of a single element $\xi \in \partial D.$ The
proof of uniform convergence on bounded sets is standard (see, e.g., \cite%
{Be2}): suppose, on the contrary, that there exist an open neighbourhood $%
U\subset \overline{D}$ of $\xi $, a bounded set $K\subset D$ and a sequence $%
\{y_{\lambda _{n}}\}\subset K$ $(\lambda _{n}\rightarrow \infty )$ such that 
$R_{\lambda _{n}}(y_{\lambda _{n}})\notin U$ for each $n$. Then 
\begin{equation*}
d(R_{\lambda _{n}}(y_{\lambda _{n}}),R_{\lambda _{n}}(y))\leq d(y_{\lambda
_{n}},y)\leq \diam K
\end{equation*}%
for any $y\in K$ and, since $R_{\lambda _{n}}(y)\rightarrow \xi $, we deduce
from Lemma \ref{Ax2} that $R_{\lambda _{n}}(y_{\lambda _{n}})\rightarrow \xi
\in \overline{D}\setminus U$, a contradiction.
\end{proof}

\bigskip 

\textbf{Acknowledgements} The first author  was partially supported by National Science Center (Poland) Preludium Grant No. UMO-2021/41/N/ST1/02968.

\end{document}